\documentclass[a4paper]{amsart}
\usepackage{tikz-cd}

\usepackage[all]{xy}

\usepackage{amsmath}
\usepackage{amssymb}
\usepackage{amsfonts}
\usepackage{amsthm}
\usepackage{amscd}

\usepackage{mathrsfs}

\usepackage{mathtools}

\usepackage{pifont}
\usepackage{latexsym}

\usepackage{graphicx} 

\usepackage{stmaryrd}

\usepackage{comment}

\usepackage{enumitem}

\newcommand{\dalgb}{\mathsf{Alg}}
\newcommand{\dalgd}{\mathsf{Alg}_{\Diamond}}

\newcommand{\frmgld}{\mathsf{Frm_{\mathsf GLD}}}

\newcommand{\vl}{\models}
\newcommand{\yy}{\rightarrow}

\newcommand{\pscom}{\mbox{$\rightarrow$}}

\newcommand{\power}{\mathcal{P}}

\newcommand{\nsystem}{\mathcal{N}}

\newcommand{\da}{\mbox{$\downarrow$}}
\newcommand{\ua}{\mbox{$\uparrow$}}

\newcommand{\thn}{\ \Rightarrow\ }
\newcommand{\eq}{\ \Leftrightarrow\ }

\newcommand{\gm}{\Gamma}
\newcommand{\dl}{\Delta}

\newcommand{\logic}[1]{\mathsf{#1}}

\newcommand{\propgl}{\logic{GL}}

\newcommand{\psgld}{\propgl_{+}^{\top\bot}}

\newcommand{\db}{\vdash}
\newcommand{\pflt}{\mathcal{F}_{p}}
\newcommand{\gldflt}{\mathcal{F}_{\mathrm{GLD}}}
\newcommand{\gldidl}{\mathcal{I}_{\mathrm{GLD}}}

\newcommand{\propvar}{\mathsf{Prop}}

\newcommand{\nqgl}{\mathsf{NQGL}}

\newcommand{\cma}{\mathcal{M}}

\newcommand{\cglm}{\mathcal{M}_{\mathsf{GL}}}
\newcommand{\cglmd}{\mathcal{M}_{\mathsf{GLD}}}
\newcommand{\cgld}{\mathcal{D}_{\mathsf{GLD}}}
\newcommand{\cmdl}{\mathcal{D}}

\newcommand{\qiop}{\mathrm{Th}_{\mathit{q}}}

\newcommand{\clhomo}{\mathit{H}}
\newcommand{\clsub}{\mathit{S}}
\newcommand{\clprod}{\mathit{P}}
\newcommand{\clprodr}{\mathit{P}_{\mathit{R}}}
\newcommand{\cliso}{\mathit{I}}
\newcommand{\clqv}{\mathit{Q}}

\newcommand{\modelop}{\mathrm{Mod}}

\newcommand{\compl}[1]{{#1}^{\mathrm{c}}}

\newcommand{\lansqcup}{f_{\sqcup}}
\newcommand{\lansqcap}{f_{\sqcap}}
\newcommand{\lanzero}{c_{0}}
\newcommand{\lanone}{c_{1}}
\newcommand{\laneq}{\approx}
\newcommand{\lanbox}{f_{\Box}}
\newcommand{\landiamond}{f_{\Diamond}}

\newcommand{\etap}{\eta_{p}}
\newcommand{\reduct}{\upharpoonright}

\newcommand{\psdunn}{\logic{K}_{+}^{\top\bot}}

\newcommand{\kfrm}{\mathsf{Frm}}

\numberwithin{equation}{section}

\theoremstyle{plain}
\newtheorem{theorem}{Theorem}[section]
\newtheorem{lemma}[theorem]{Lemma}

\newtheorem{corollary}[theorem]{Corollary}

\theoremstyle{definition}
\newtheorem{definition}[theorem]{Definition}

\date{}

\begin{document}
\title[The positive fragment of provability logic]
{The positive fragment of provability logic: 
An infinitary proof system and algebraic representation} 

\author{Yoshihito Tanaka}
\address{Kyushu Sangyo University}

\thanks{
The author would like to thank Taishi Kurahashi for his helpful comments and suggestions. 
}

\begin{abstract}
We introduce an infinitary proof system $\psgld$ 
for the positive fragment of the provability logic $\propgl$. 
The system extends Dunn's sequent system $\psdunn$ 
for positive modal logic with two $\omega$-rules 
designed to capture the characteristic behavior of $\propgl$ 
in a language containing neither negation nor implication. 
We establish the completeness of $\psgld$ by algebraic methods. 
More precisely, we define a class $\cgld$ of 
modal distributive lattices, called \emph{GLD-lattices}, 
whose defining conditions correspond to the two $\omega$-rules. 
We call a modal algebra a \emph{GLD-algebra} if its 
modal-distributive-lattice reduct is a GLD-lattice, 
and prove that every GLD-lattice embeds into a GLD-algebra. 
Using this embedding theorem together with 
Mal'cev's characterization of quasivarieties, 
we show that GLD-lattices and GLD-algebras validate 
exactly the same quasi-identities of modal distributive lattices. 
It follows that, for all positive modal formulas $\phi$ and $\psi$, 
the sequent $\phi\db\psi$ is provable in $\psgld$ if and only if 
$\phi\supset\psi$ is a theorem of $\propgl$. 
\end{abstract}

\maketitle

\section{Introduction}

In this paper, we present a proof system $\psgld$
for the positive fragment of $\propgl$,  
which is based on 
the sequent system 
$\psdunn$ introduced by Dunn \cite{dnn95}.  
Positive modal formulas are modal formulas containing 
neither negation nor implication. 
More precisely, they are modal formulas built from 
propositional variables, $\bot$, $\top$, 
$\lor$, $\land$, $\Diamond$, and $\Box$.  
The logic $\propgl$ is the least normal modal logic extending 
$\logic{K}$ and containing the L\"{o}b formula 
$\Box(\Box p\supset p)\supset\Box p$. 
As in Dunn \cite{dnn95}, 
a sequent is 
an expression of the form $\phi\db\psi$, 
where $\phi$ and $\psi$ are positive 
modal formulas. 
We prove that, for all positive modal formulas $\phi$ and $\psi$,
the sequent $\phi\db\psi$ is provable in $\psgld$ if and only if 
$\phi\supset\psi$ is provable in $\propgl$.

Proof systems for positive modal formulas 
in the style of Dunn \cite{dnn95} 
have been developed 
for several modal logics.
Dunn \cite{dnn95} provided a proof system for 
the positive fragment of 
$\logic{K}$
and proved its Kripke completeness, also   
noting that the result can be extended to 
$\logic{T}$, 
$\logic{B}$, 
$\logic{S4}$, 
$\logic{S5}$, 
$\logic{KB}$, 
$\logic{K4}$, and
$\logic{K5}$.  
While $\Box$ and $\Diamond$ in Dunn \cite{dnn95} have dual properties,    
Celani and Jansana \cite{cln-jns97} treated these operators 
independently. 
They presented
sequent systems for
$\logic{K}$, 
$\logic{T}$, 
$\logic{4}$, 
$\logic{B}$, 
$\logic{S}$, 
$\logic{E}$, and
$\logic{D}$, 
in which $\Box$ and $\Diamond$ are independent, 
and proved the
completeness theorems for these systems with respect to Kripke frames for 
intuitionistic modal logics. 
As for $\propgl$, 
it was proved by Dashkov \cite{dsh12} that 
the strictly positive fragment of $\propgl$, namely the
formulas constructed from $\land$, $\Diamond$, $\top$, 
and propositional variables, 
coincides with that of $\logic{K4}$. 
However, the positive fragment of $\propgl$ strictly contains 
that of $\logic{K4}$, because, for example, it 
contains $\Box p\lor\Diamond\Box p$, in particular, $\Box\bot\lor\Diamond\Box\bot$, which
expresses the second incompleteness theorem of G\"{o}del. 
To the best of our knowledge, no proof system for the positive fragment of $\propgl$ 
has appeared in the literature.

Our proof system contains $\omega$-rules, that is, 
inference rules that have countably many premises. 
In Tanaka \cite{tnk18}, a Gentzen system $\nqgl$
for 
a predicate extension of $\propgl$ was introduced 
and its Kripke completeness was proved. 
The system $\nqgl$ does not include the L\"{o}b formula as an 
axiom, but has
an $\omega$-rule of the following form: 
\begin{equation}\label{sequentomega}
\dfrac{\gm\yy\dl,\Diamond^{n}\top,\ \text{($\forall n\in\omega$)}}
{\gm\yy\dl}. 
\end{equation}
As \eqref{sequentomega} contains neither negation symbols nor 
implication symbols, we make use of a variation of \eqref{sequentomega} to 
define our proof system $\psgld$.

We use methods from equational logic and universal algebras to prove the 
completeness of $\psgld$.  
We introduce a class $\cgld$ of modal distributive lattices, 
called \emph{GLD-lattices}, 
whose defining conditions algebraically 
encode \eqref{sequentomega}. 
A class of modal algebras  
whose modal-distributive-lattice reduct is in $\cgld$ was 
introduced by Sasaki and Tanaka \cite{ssk-tnk24}. 
We call these modal algebras \emph{GLD-algebras}. 
It was essentially proved by Sasaki and Tanaka \cite{ssk-tnk24}
that $\propgl$ is sound and complete with respect to the 
class of GLD-algebras. 
Next, we show that 
for every $A\in\cgld$, there 
exists a GLD-algebra $B$ 
together with
a monomorphism $f:A\yy B$ of modal distributive lattices.  
Using the embedding theorem together with Mal'cev's characterization of 
quasivarieties \cite{mlc66}, 
we prove that GLD-lattices and 
GLD-algebras validate exactly the same quasi-identities of 
modal distributive lattices. 
Combining this result with the algebraic completeness of $\psgld$ 
with respect to $\cgld$, we obtain the desired completeness theorem for 
the positive fragment of $\propgl$. 
Similar completeness proofs for strictly positive modal logics 
using a related algebraic method 
can be found in Sofronie-Stokkermans \cite{stokkermans08} and in 
Kikot, Kurucz, Tanaka, Wolter, and Zakharyaschev 
\cite{kkt-krc-tnk-wlt-zkh}. 
Since $\psgld$ has $\omega$-rules, the above monomorphism 
needs to preserve certain infinite joins and meets. 
For this purpose, we use the \emph{GLD-filters} 
introduced by Sasaki and Tanaka \cite{ssk-tnk24} 
instead of prime filters.

We then apply the above method to establish completeness for 
the positive fragments of several other modal logics, namely  
those 
characterized by Kripke frames satisfying combinations of 
reflexivity, transitivity, and symmetry.  
We introduce proof systems for the positive fragments of these 
modal logics 
and prove their Kripke completeness. 
This gives explicit proofs of 
some results
noted in Dunn \cite{dnn95}. 
In this part of the argument, we use prime filters rather than GLD-filters.

In the final part of the paper, 
we introduce {\em positive modal quasi-sequents}, 
that is, 
a logical counterpart 
of quasi-identities, and extend 
the completeness theorem for sequents 
mentioned above to positive modal quasi-sequents.

The remainder of the paper is organized as follows. 
In Section \ref{sec:gld}, we introduce GLD-lattices and GLD-filters and 
recall the algebraic preliminaries used in the subsequent sections. 
In Section \ref{sec:emb}, we prove that every GLD-lattice embeds 
into a GLD-algebra. 
In Section \ref{sec:eqlog}, we use this embedding to compare the quasi-identities 
of modal distributive lattices 
valid in GLD-lattices with those valid in GLD-algebras. 
In Section \ref{sec:proofsystem}, we define the 
infinitary proof system $\psgld$ and establish its soundness 
and completeness for the positive fragment of $\propgl$. 
In Section \ref{sec:applications}, we apply the same method 
to the positive fragments of several standard modal logics 
and obtain algebraic proofs of the corresponding Kripke completeness results.
Finally, in Section \ref{sec:quasi-sequent}
we present a completeness theorem for positive modal quasi-sequents. 

\section{GLD-lattices and GLD-filters}\label{sec:gld} 

In this section, 
we introduce GLD-lattices and GLD-filters. 
We also clarify the notation and 
conventions used throughout this paper, and 
recall some facts for later use.

For each set $A$ and $S\subseteq A$, we write $\compl{S}$ 
for $A\setminus S$. 
For each ordered set $\langle A,\leq\rangle$ and $S\subseteq A$, 
we write $\ua S$ for the upward closure of $S$. 
Let $A$ and $B$ be any sets and 
$R\subseteq A\times B$. 
For any set $X\subseteq A$ and $Y\subseteq B$, 
we write $R\left[X\right]$ for the set
$\{y\in B\mid \exists x\in X\left((x,y)\in R\right)\}$, 
and write 
$R^{-1}\left[Y\right]$ for the set 
$\{x\in A\mid\exists y\in Y \left((x,y)\in R\right)\}$. 
If $X=\{x\}$, we write $R[x]$ for $R[X]$, and 
if $Y=\{y\}$, we write $R^{-1}[y]$ for $R^{-1}[Y]$.

%
%
%

\begin{definition}\label{defmodaldl}
An algebra $\langle A;\sqcup,\sqcap,\Diamond,\Box,0,1\rangle$ is called 
a \emph{modal distributive lattice}
if its reduct $\langle A;\sqcup,\sqcap,0,1\rangle$ is a 
distributive lattice 
and $\Box$ and $\Diamond$ are unary operators on $A$ 
satisfying the following: 
\begin{enumerate}
\item\label{minmax}
$\Diamond 0=0$ and $\Box 1=1$; 
\item\label{distributivity}
$\Diamond (x\sqcup y)= \Diamond x\sqcup\Diamond y$ and 
$\Box (x\sqcap y)= \Box x\sqcap\Box y$
for every $x$ and $y$ in $A$;  
\item\label{boxdia}
$\Box (x\sqcup y)\leq \Box x\sqcup\Diamond y$  and  
$\Box x\sqcap\Diamond y\leq\Diamond (x\sqcap y)$ 
for all $x$ and $y$ in $A$.
\end{enumerate}
We write $\cmdl$ for the class of all modal distributive lattices.

Let $A$ and $B$ be modal distributive lattices. A function 
$f:A\yy B$ is called a \emph{homomorphism of modal distributive lattices} 
if it is a homomorphism of bounded lattices and 
satisfies $f(\Diamond x)=\Diamond f(x)$ and 
$f(\Box x)=\Box f(x)$ for every $x\in A$.  
A homomorphism of modal distributive lattices 
is called a \emph{monomorphism} if it is injective 
and is called an \emph{isomorphism} if it is injective and surjective. 
\end{definition}

Modal distributive lattices were introduced by Dunn \cite{dnn95} 
under the name \emph{positive modal algebras}. 
Dunn also considers versions without $0$ and $1$, 
corresponding to languages without $\bot$ and $\top$. 
Since we use the term \emph{modal algebra} 
for the Boolean-algebra-based structures of Definition \ref{modal-algebra}, 
we use the term \emph{modal distributive lattice} to avoid ambiguity.

\begin{definition}\label{defgldl}
A modal distributive lattice $A$ is called 
a \emph{GLD-lattice}
if it
satisfies the following: 
\begin{enumerate}
\item\label{gldltransitivity}
$\Diamond \Diamond x\leq\Diamond x$ and 
$\Box x\leq\Box\Box x$
for every $x$ in $A$;   
\item\label{ioriand}
$\bigsqcap_{n\in\omega}\left(x\sqcup\Diamond^{n}1\right)= x$ and 
$\bigsqcup_{n\in\omega}\left(x\sqcap\Box^{n}0\right)= x$ 
for every $x$ in $A$. 
\end{enumerate}
We write $\cgld$ for the class of all GLD-lattices. 
\end{definition}

\begin{definition}
Let $A$ be a modal distributive lattice. A non-empty set $F\subseteq A$ is called 
a \emph{filter} of $A$ if it satisfies the following: 
\begin{enumerate}
\item
$\ua F=F$;
\item
for every $x$ and $y$ in $A$,  
if $x$ and $y$ are in $F$, then $x\sqcap y\in F$. 
\end{enumerate}
A filter $F$ is said to be \emph{prime} if $0\not\in F$ and 
for every 
$x\sqcup y\in F$, either $x\in F$ or $y\in F$. 
A prime filter $F$ is called a \emph{GLD-filter} if 
there exists $n\in\omega$ such that $\Diamond^{n}1\not\in F$ and 
$\Box^{n}0\in F$. 
We write $\gldflt(A)$ for the set of 
all GLD-filters of $A$.  
%
%
%
%
For each modal distributive lattice $A$ and $x\in A$, 
we define $\eta(x)\subseteq\gldflt(A)$
%
%
by 
$
\eta(x)=
\left\{
F\in\gldflt(A)\mid x\in F
\right\}.
%
%
$
\end{definition}

The following theorem is a variant of the 
prime filter theorem for GLD-filters.

\begin{theorem}\label{gldflttheorem}
Let $A$ be a GLD-lattice. 
Suppose that $x$ and $y$ are in $A$ and $x\not\leq y$. 
Then, there exists a GLD-filter $F$ of $A$ 
%
%
such that 
$x\in F$ and $y\not\in F$. 
%
%
\end{theorem}

\begin{proof}
By \eqref{ioriand} of Definition \ref{defgldl}, 
there exist $k$ and $l$ in $\omega$ such that 
$x\sqcap\Box^{k}0\not\leq y\sqcup \Diamond^{l}1$. 
Let $n=\max\{k,l\}$. 
Since 
$\Diamond^{n}1\leq\Diamond^{l}1$ and 
$\Box^{k}0\leq\Box^{n}0$,  
\begin{equation*}\label{witness}
x\sqcap\Box^{n}0
\not\leq
y\sqcup\Diamond^{n}1. 
\end{equation*}
Then, by the prime filter theorem, 
there exists a prime filter $F$ of $A$ such that 
$x\sqcap\Box^{n}0\in F$ and 
$y\sqcup\Diamond^{n}1\not\in F$. 
It is obvious that $F$ is a GLD-filter. 
%
%
\end{proof}

Next, we introduce 
Boolean-algebra-based structures.

\begin{definition}\label{modal-algebra}
An algebra $\langle A;\sqcup,\sqcap,-,\Box,0,1\rangle$ is called 
a \emph{modal algebra}
if its reduct $\langle A;\sqcup,\sqcap,-,0,1\rangle$ is a Boolean algebra 
and $\Box$ is a unary operator on $A$ satisfying the 
following conditions: 
\begin{enumerate}
\item\label{prtp}
$\Box 1=1$; 
\item\label{prmt}
$\Box (x\sqcap y)= \Box x\sqcap\Box y$
for every $x$ and $y$ in $A$. 
\end{enumerate} 
For each $x$ and $y$ in $A$, we write $\Diamond x$  for 
$-\Box-x$ and 
$x\pscom y$ for $-x\sqcup y$. 
We write $\cma$ for the class of all modal algebras. 
For each modal algebra $A$, we write $A\reduct_{\Sigma_{\cmdl}}$
for the modal distributive 
lattice reduct of $A$, 
and for each class $C$ of modal algebras, 
we write $C\reduct_{\Sigma_{\cmdl}}$ for $\{A\reduct_{\Sigma_{\cmdl}}\mid A\in C\}$.
\end{definition}

Some papers, for example Do\v{s}en \cite{dsn89}, 
define a modal algebra simply as a Boolean algebra equipped with a unary operator 
$\Box$, without requiring any axioms for that operator. 
In this paper, however, we assume normality, since our focus is on $\propgl$.

\begin{definition}\label{neighborhoodframes}
A \emph{neighborhood frame} is a pair
$\langle W, \nsystem\rangle$, where
$W$ is a non-empty set and 
$\nsystem$ is a function from $W$ to  $\power(\power(W))$. 
For each neighborhood frame $Z=\langle W,\nsystem\rangle$, 
the \emph{complex algebra} of $Z$, which is denoted by
$\dalgb(Z)$,  is 
$$
\dalgb(Z)=
\langle
\power(W);\cup,\cap,\compl{(-)},\Box_{Z},\emptyset,W
\rangle, 
$$
where
$$
\Box_{Z}X=\{w\in W\mid X\in\nsystem(w)\}
$$
for any $X\subseteq W$. 


%
%
\end{definition}

A Kripke frame $\langle W,R\rangle$ is equivalent to the 
neighborhood frame $\langle W,\nsystem\rangle$ such that 
$\nsystem(w)=\ua R[w]$ for every $w\in W$.

\begin{theorem}
[Do\v{s}en \cite{dsn89}]\label{complexnormal}
Let $Z=\langle W,\nsystem\rangle$ be any neighborhood frame. 
The complex algebra $\dalgb(Z)$ is a modal algebra 
if and only if $\nsystem(w)$ is a filter of $\power(W)$
for every $w\in W$. 
\end{theorem}

\begin{definition}\label{defglalg}
A modal algebra $A$ is called a 
\emph{GL-algebra}
if it satisfies 
$$
\Box (\Box x\pscom x)\leq \Box x
$$
for every $x$ in $A$.  
We write $\cglm$ for the class of all GL-algebras.
\end{definition}

It is straightforward to verify the following theorem.

\begin{theorem}\label{propglcompleteness}
$\propgl$ is sound and complete with respect to $\cglm$. 
\end{theorem}

\begin{definition}\label{defgldalg}
A modal algebra $A$ is called a \emph{GLD-algebra} if it 
satisfies the following: 
\begin{enumerate}
\item\label{transitivity}
$\Diamond\Diamond x\leq\Diamond x$ for every $x\in A$;
\item\label{diamondstar}
$\bigsqcap_{n\in\omega}{\Diamond}^{n}1=0$. 
\end{enumerate}
\end{definition}

It is straightforward to verify that if $A$ is a GLD-algebra 
then $A\reduct_{\Sigma_{\cmdl}}\in\cgld$. 
The following theorem was essentially proved by 
Sasaki and Tanaka \cite{ssk-tnk24} by a proof-theoretic method. 
However, by Theorem \ref{gldflttheorem}, we have a concise algebraic proof. 

\begin{theorem}
[Sasaki and Tanaka \cite{ssk-tnk24}]\label{sasaki-tanaka} 
$\cglmd$ is a proper subclass of $\cglm$. 
\end{theorem}

\begin{proof}
First we show $\cglmd\subseteq\cglm$. 
Let $A$ be a GLD-algebra. Suppose 
$$
\Box(\Box x\pscom x)\not\leq \Box x
$$ 
for some $x\in A$.
By Theorem \ref{gldflttheorem}, there exists a GLD-filter $F$ 
such that $\Box(\Box x\pscom x)\in F$ and $\Box x\not\in F$. 
Then, there exists a least $n\geq 1$ such that $\Box^{n+1}x\in F$ 
and $\Box^{n}x\not\in F$ since $F$ is a GLD-filter. 
By \eqref{transitivity} of Definition \ref{defgldalg}, 
$\Box^{n}(\Box x\pscom x)\in F$. Hence, 
$$
\Box^{n}(\Box x\pscom x)
\leq
\Box^{n+1}x\pscom \Box^{n}x\in F. 
$$
Then, $\Box^{n}x\in F$, but this is a contradiction. 
Hence, $A$ is a GL-algebra. 
We show $\cglmd\not=\cglm$. 
Let $Z=\langle\omega+1,R\rangle$ be a Kripke frame 
such that $(n,m)\in R$ if and only if $m<n$. 
Then $\dalgb(Z)$ 
is a GL-algebra, but not 
a GLD-algebra. 
\end{proof}

The following theorem was also essentially proved by 
Sasaki and Tanaka \cite{ssk-tnk24}, but we present 
an explicit proof here.

\begin{theorem}
[Sasaki and Tanaka \cite{ssk-tnk24}] \label{propglcompletenessglda}
$\propgl$ is sound and complete with respect to the 
class of GLD-algebras. 
\end{theorem}

\begin{proof}
Soundness follows from Theorems \ref{propglcompleteness} and \ref{sasaki-tanaka}. 
We show completeness. Suppose $\phi$ is not in $\propgl$. Then, 
there exists a finite, transitive, and irreflexive Kripke frame $Z$ such that $Z\not\vl\phi$. 
It is easy to verify that $\dalgb(Z)\in\cglmd$. 
\end{proof}

\section{Embedding GLD-lattices into GLD-algebras}\label{sec:emb}

In this section, we show that every 
GLD-lattice embeds into some 
GLD-algebra.

\begin{definition}\label{gldfltfrm}
Let $A$ be a GLD-lattice. 
We say that $Z=\langle\gldflt(A),\nsystem\rangle$
is a \emph{GLD-filter frame} of $A$ if $\nsystem(F)$ is a 
filter of $\power(\gldflt(A))$ and 
satisfies the following
for every $F\in\gldflt(A)$:  
\begin{enumerate}
\item\label{etasubset}
$
\left\{
\eta(x)\mid \Box x\in F
\right\}\subseteq \nsystem(F)
$ 
and 
$
\left\{
\compl{\eta(x)}\mid \Diamond x\not\in F
\right\}\subseteq \nsystem(F)
$;
\item\label{etaemptyset}
$
\left\{
\eta(x)\mid \Box x\not\in F
\right\}
\subseteq
\compl{\nsystem(F)}
$ 
and
$
\left\{
\compl{\eta(x)}\mid \Diamond x\in F
\right\}
\subseteq
\compl{\nsystem(F)}
$.
\end{enumerate}

%
%
\end{definition}

\begin{theorem}\label{embedding}
Let $A$ be a GLD-lattice.  
If $Z=\langle\gldflt(A),\nsystem\rangle$ 
is a GLD-filter frame,
then $\eta:A\yy\dalgb(Z)\reduct_{\Sigma_{\cmdl}}$ is a monomorphism of modal distributive lattices. 
%
%
\end{theorem}

\begin{proof}
It is straightforward to show that $\eta$ is a homomorphism of bounded
lattices. 
By Theorem \ref{gldflttheorem}, $\eta$ is injective. 
Let $x\in A$. Then, for every $F\in\gldflt(A)$, we have 
\begin{align*}
F\in\eta(\Box x)
&\eq
\Box x\in F\\
&\eq
\eta(x)\in\nsystem(F)
& (\text{Definition \ref{gldfltfrm}})\\
&\eq
F\in\Box_{Z}\eta(x)
& (\text{Definition \ref{neighborhoodframes}})
\end{align*}
and 
\begin{align*}
F\in\eta(\Diamond x)
&\eq
\Diamond x\in F\\
&\eq
\compl{\eta(x)}\not\in\nsystem(F)
& (\text{Definition \ref{gldfltfrm}})\\
&\eq
F\not\in\Box_{Z}\compl{\eta(x)}
& (\text{Definition \ref{neighborhoodframes}})\\
&\eq
F\in-\Box_{Z}-\eta(x)
& (\text{Definition \ref{neighborhoodframes}})\\
&\eq
F\in\Diamond_{Z}\eta(x). 
\end{align*}
Hence, $\eta(\Box x)=\Box_{Z}\eta(x)$ and 
$\eta(\Diamond x)=\Diamond_{Z}\eta(x)$ 
for every $x\in A$. 
%
%
\end{proof}

\begin{theorem}\label{cgl-cgld}
Suppose that $A$ is a GLD-lattice. 
For each $F\in\gldflt(A)$,  
define
\begin{equation}\label{canonical}
\nsystem(F)
=
\ua\left\{
\eta(x)\cap\compl{\eta(y)}\mid 
\text{$\Box x\in F$ and $\Diamond y\not\in F$}
\right\}. 
\end{equation}
Then, 
\begin{enumerate}
\item
$Z=\langle \gldflt(A),\nsystem\rangle$ 
is a GLD-filter frame of $A$; 
\item
$\dalgb(Z)$ is a GLD-algebra. 
\end{enumerate}
%
%
\end{theorem}

\begin{proof}
(1): 
%
%
First, we show that 
$\nsystem(F)$ is a filter of $\power(\gldflt(A))$ for 
every GLD-filter $F$. 
Suppose that $X_{1}$ and $X_{2}$ are in $\nsystem(F)$. 
Then, there exist $\Box x_{1}$ and $\Box x_{2}$ in $F$ 
and $\Diamond y_{1}$ and $\Diamond y_{2}$ in $\compl{F}$ 
such that  
\begin{equation*}
\eta(x_{1})\cap\compl{\eta(y_{1})}\subseteq X_{1},\ 
\eta(x_{2})\cap\compl{\eta(y_{2})}\subseteq X_{2}. 
\end{equation*}
Then, 
\begin{align*}
X_{1}\cap X_{2}
&\supseteq
\eta(x_{1})\cap\compl{\eta(y_{1})}\cap
\eta(x_{2})\cap\compl{\eta(y_{2})}\\
&=
\eta(x_{1}\sqcap x_{2})\cap
\compl{\eta(y_{1}\sqcup y_{2})}. 
\end{align*}
Since 
$\Box(x_{1}\sqcap x_{2})=\Box x_{1}\sqcap\Box x_{2}\in F$ 
and 
$\Diamond(y_{1}\sqcup y_{2})=\Diamond y_{1}\sqcup\Diamond y_{2}\not\in F$,
we have $X_{1}\cap X_{2}\in\nsystem(F)$. 
Hence, $\nsystem(F)$ is a filter of $\power(\gldflt(A))$ for every $F\in\gldflt(A)$. 
It is straightforward to see 
that \eqref{etasubset} of Definition \ref{gldfltfrm} 
follows from \eqref{canonical}. 
We show \eqref{etaemptyset} of Definition \ref{gldfltfrm}. 
First, suppose that there exists 
$X\in \left\{
\eta(x)\mid \Box x\not\in F
\right\}\cap\nsystem(F)
$. 
Then, there exist $\Box x\not\in F$, $\Box u\in F$, and $\Diamond v\not\in F$ 
such that $X=\eta(x)$ and 
$$
\eta(u)\cap\compl{\eta(v)}\subseteq\eta(x). 
$$
Hence, 
\begin{equation}\label{empty1}
\eta(u)\subseteq \eta(x)\cup\eta(v)=\eta(x\sqcup v). 
\end{equation}
We claim that $u\not\leq x\sqcup v$. If not, it follows from 
\eqref{boxdia} of Definition \ref{defmodaldl} that 
$$
\Box u\leq\Box(x\sqcup v)\leq \Box x\sqcup\Diamond v. 
$$
This contradicts $\Box u\in F$, since 
$\Box x\not\in F$ and $\Diamond v\not\in F$.   
This completes the proof of the claim. 
By Theorem \ref{gldflttheorem}, 
there exists a GLD-filter $G$ such that 
$u\in G$ and $x\sqcup v\not\in G$. 
This contradicts \eqref{empty1}. 
Next, suppose that there exists 
$X\in \left\{
\compl{\eta(x)}\mid \Diamond x\in F
\right\}\cap\nsystem(F)
$. 
Then, there exist $\Diamond x\in F$, $\Box u\in F$, and $\Diamond v\not\in F$ 
such that $X=\compl{\eta(x)}$ and 
$$
\eta(u)\cap\compl{\eta(v)}\subseteq\compl{\eta(x)}. 
$$
Hence, 
\begin{equation}\label{empty2}
\eta(x\sqcap u)
= 
\eta(x)\cap\eta(u)
\subseteq\eta(v)
\end{equation}
We claim that $x\sqcap u\not\leq v$. If not, it follows from 
\eqref{boxdia} of Definition \ref{defmodaldl} that 
$$
\Diamond x\sqcap\Box u
\leq
\Diamond(x\sqcap u)
\leq\Diamond v. 
$$
This contradicts $\Diamond v\not\in F$, since 
$\Diamond x\in F$ and $\Box u\in F$.   
This completes the proof of the claim. 
By Theorem \ref{gldflttheorem}, 
there exists a GLD-filter $G$ such that 
$x\sqcap u\in G$ and $v\not\in G$. 
This contradicts \eqref{empty2}.

(2):  
%
%
By Theorem \ref{complexnormal}, 
$\dalgb(Z)$ is a modal algebra. 
We show \eqref{transitivity} of Definition \ref{defgldalg}. 
Let $X\in\power(\gldflt(A))$. For
every $F\in\gldflt(A)$, 
\begin{align*}
F\in\Box_{Z}X
&\eq
X\in\nsystem(F)\\
&\eq
\exists\Box x\in F\exists\Diamond y\not\in F
\left(
\eta(x)\cap\compl{\eta(y)}\subseteq X
\right)\\
&\thn
\exists\Box\Box x\in F\exists\Diamond\Diamond y\not\in F
\left(
\eta(x)\cap\compl{\eta(y)}\subseteq X
\right)\\
&\thn
\exists\Box\Box x\in F\exists\Diamond\Diamond y\not\in F
\left(
\Box_{Z}
\left(
\eta(x)\cap\compl{\eta(y)}
\right)
\subseteq 
\Box_{Z}X
\right)\\
&\eq
\exists\Box\Box x\in F\exists\Diamond\Diamond y\not\in F
\left(
\eta(\Box x)\cap\compl{\eta(\Diamond y)}
\subseteq 
\Box_{Z}X
\right)\\
&\thn
F\in\Box_{Z}\Box_{Z}X. 
\end{align*}
Here, the third implication holds by 
\eqref{gldltransitivity} of Definition \ref{defgldl}, 
and the fourth implication holds because $\Box_{Z}$ is monotone
because  $\dalgb(Z)$ is a modal algebra. 
The fifth equivalence follows from the following: 
\begin{align*}
\Box_{Z}
\left(
\eta(x)\cap\compl{\eta(y)}
\right)
&=
\Box_{Z}
\left(
\eta(x)\cap-\eta(y)
\right)
& (\text{Definition \ref{neighborhoodframes}})\\
&=
\Box_{Z}\eta(x)\cap\Box_{Z}-\eta(y)
&\text{(Definition \ref{modal-algebra})}
\\
&=
\Box_{Z}\eta(x)\cap\compl{\left(\Diamond_{Z}\eta(y)\right)}
& (\text{Definition \ref{neighborhoodframes}})\\
&=
\eta(\Box x)\cap\compl{\eta(\Diamond y)}.
&\text{(Theorem \ref{embedding})} 
\end{align*}
We show \eqref{diamondstar} of Definition \ref{defgldalg}. 
By Theorem \ref{embedding}, 
$$
\bigcap_{n\in\omega}{\Diamond_{Z}}^{n}\gldflt(A)
=
\bigcap_{n\in\omega}{\Diamond_{Z}}^{n}\eta(1)
=
\bigcap_{n\in\omega}\eta(\Diamond^{n}1). 
$$
Hence, if $F\in \bigcap_{n\in\omega}{\Diamond_{Z}}^{n}\gldflt(A)$
then $\Diamond^{n}1\in F$ for all $n\in\omega$. 
However, this is a contradiction, because $F$ is a GLD-filter. 
%
%
\end{proof}

\begin{theorem}\label{cgldembcgl}
For every GLD-lattice $A$, there exists a GLD-algebra $B$ and
a monomorphism $\eta:A\yy B\reduct_{\Sigma_{\cmdl}}$ of modal distributive lattices. 
\end{theorem}

\begin{proof}
Let $A\in\cgld$. 
Define $Z=\langle \gldflt(A),\nsystem\rangle$, where
$\nsystem$ 
is given by 
\eqref{canonical}. 
Then, $\dalgb(Z)$ is a GLD-algebra by Theorem \ref{cgl-cgld}, 
and $\eta: A\yy\dalgb(Z)\reduct_{\Sigma_{\cmdl}}$ is a monomorphism of modal distributive lattices 
by 
Theorems \ref{embedding} and \ref{cgl-cgld}.  
\end{proof}

\section{Quasi-identities of modal distributive lattices}\label{sec:eqlog}

In this section, 
we show that for every quasi-identity $\alpha$ of 
modal distributive lattices, 
$\alpha$ is valid in $\cgld$
if and only if it is valid in $\cglmd\reduct_{\Sigma_{\cmdl}}$. 
We fix the first-order signature of modal distributive lattices. 
It contains constant symbols $\lanzero$ and $\lanone$, 
unary function symbols $\landiamond$ and $\lanbox$, 
and binary function symbols $\lansqcup$ and $\lansqcap$. 
These symbols are interpreted as $0$, $1$, $\Diamond$, $\Box$, $\sqcup$, 
and $\sqcap$, respectively. 
We use $\laneq$ for equality between terms.
Although the terminology below is standard for algebras of an 
arbitrary signature, we formulate the definitions specifically 
for modal distributive lattices.

\begin{definition}
We fix a countable set $X$ of variables. 
The set $T$ of \emph{terms} of modal distributive lattices is the 
least set such that 
\begin{enumerate}
\item
$X\subseteq T$; 
\item
$\lanzero\in T$ and $\lanone\in T$; 
\item
if $s\in T$ then $\lanbox(s)$ and $\landiamond(s)$ are in $T$; 
\item
if $s$ and $t$ are in $T$ then $\lansqcup(s,t)$ and $\lansqcap(s,t)$ are in $T$. 
\end{enumerate}
An \emph{identity} of modal distributive lattices 
is a first-order formula of the form $s\laneq t$, where $s$ and $t$ 
are terms of modal distributive lattices.  
A \emph{quasi-identity} of modal distributive lattices 
is a first-order formula of the form 
$s_{1}\laneq t_{1}\land\cdots\land s_{n}\laneq t_{n}\supset s\laneq t$, 
where $s_{1}\laneq t_{1},\dots,s_{n}\laneq t_{n}$ and $s\laneq t$ 
are identities of modal distributive lattices.  
It is clear that an identity is a quasi-identity. 
\end{definition}

Let $s$ be a term. 
We write $s$ as $s(x_{1},\ldots,x_{n})$ if  
every variable occurring in $s$ belongs to $x_{1},\ldots,x_{n}$. 
For each term $s(x_{1},\ldots,x_{n})$ of modal distributive lattices, 
each modal distributive lattice $A$, and each $a_{1},\ldots,a_{n}$ in $A$, 
we define 
$s^{A}(a_{1},\ldots,a_{n})\in A$ in the usual way.

\begin{definition}
Let $A$ be a 
modal distributive lattice and 
$s\laneq t$ be an identity of modal distributive lattices. 
We write $A\vl s\laneq t$
if 
$s^{A}(a_{1},\ldots,a_{n})=t^{A}(a_{1},\ldots,a_{n})$
holds
for all $a_{1},\ldots,a_{n}$ in $A$. 
Let 
$s_{1}\laneq t_{1}\land\cdots\land s_{m}\laneq t_{m}\supset s\laneq t$
be a quasi-identity of modal distributive lattices. 
We write $A\vl s_{1}\laneq t_{1}\land\cdots\land s_{m}\laneq t_{m}\supset s\laneq t$  
if $A$ satisfies the following: 
for every $a_{1},\ldots, a_{n}$ in $A$, if 
$s^{A}_{i}(a_{1},\ldots,a_{n})=t^{A}_{i}(a_{1},\ldots,a_{n})$ 
for every $i=1,\ldots, m$, 
then 
$s^{A}(a_{1},\ldots,a_{n})=t^{A}(a_{1},\ldots,a_{n})$.  
Let $S$ be a set of quasi-identities of modal distributive 
lattices. 
We write $A\vl S$ if 
$A\vl\alpha$ holds for every $\alpha\in S$. 
Let $C$ be a class of modal distributive lattices. 
We write $C\vl S$ if $A\vl S$ for every $A\in C$.  
\end{definition}

\begin{definition}
Let $C$ be a class of modal distributive lattices. 
We write 
$\cliso(C)$, 
$\clhomo(C)$, 
$\clsub(C)$, 
$\clprod(C)$, 
and
$\clprodr(C)$
for the least class of modal distributive lattices
containing $C$ as a subset and closed under 
isomorphic images, 
homomorphic images, subalgebras, 
products, and reduced products, respectively. 
We also write $\qiop(C)$ for the set of all 
quasi-identities $\alpha$ such that $C\vl\alpha$. 
For each set $S$ of quasi-identities of modal distributive 
lattices, we write $\modelop(S)$ for the class of all 
modal distributive lattices $A$ such that $A\vl S$. 
\end{definition}

\begin{definition}
Let $C$ be a class of modal distributive lattices. 
$C$ is said to be a \emph{variety} of 
modal distributive lattices if there exists 
a set $S$ of identities of modal distributive lattices 
such that 
$C=\modelop(S)$. 
If there exists 
a set $S$ of quasi-identities of modal distributive lattices 
such that 
$C=\modelop(S)$, 
then $C$ is said to be a \emph{quasivariety}.  
We write $\clqv(C)$ for the least quasivariety 
containing $C$ as a subset.  
\end{definition}

The following theorem 
holds for the class of algebras of any signature.

\begin{theorem}
[Mal'cev \cite{mlc66}]
\label{malcevth} 
Let $C$ be a class of algebras. Then, 
$\clqv(C)=\cliso\clsub\clprodr(C)$. 
\end{theorem}

Now, we present the main theorem of this section.

\begin{theorem}\label{qiequivalence}
For any quasi-identity $\alpha$ of modal distributive lattices,  
$\cglmd\reduct_{\Sigma_{\cmdl}}\vl \alpha$ if and only if $\cgld\vl \alpha$. 
\end{theorem}

\begin{proof}
The if-part follows since $\cglmd\reduct_{\Sigma_{\cmdl}}\subseteq\cgld$. 
We show the converse implication. 
By Theorem \ref{malcevth} and Theorem \ref{cgldembcgl},  
$$
\cgld
\subseteq
\cliso\clsub(\cglmd\reduct_{\Sigma_{\cmdl}})
\subseteq
\cliso\clsub\clprodr(\cglmd\reduct_{\Sigma_{\cmdl}})
=
\clqv(\cglmd\reduct_{\Sigma_{\cmdl}}). 
$$
Since 
$$
\clqv(\cglmd\reduct_{\Sigma_{\cmdl}})=\modelop(\qiop(\cglmd\reduct_{\Sigma_{\cmdl}})),
$$  
$\cglmd\reduct_{\Sigma_{\cmdl}}\vl \alpha$ implies $\cgld\vl \alpha$.   
\end{proof}

\section{A proof system for the positive fragment of $\propgl$}\label{sec:proofsystem}

In this section, we present a proof system for the 
positive fragment of $\propgl$. 
Here, $\propgl$ is  
the least normal modal logic extending 
$\logic{K}$ and containing the L\"{o}b formula 
$\Box(\Box p\supset p)\supset \Box p$, 
and a formula is said to be \emph{positive} if it 
contains neither negation nor implication.

\begin{definition}
The \emph{language of positive modal logic} consists
of the following symbols:  
\begin{enumerate} 
\item
a countable set $\propvar$ of propositional variables;

\item
logical constants: $\bot$ and $\top$;

\item logical connectives: 
$\lor$,  $\land$;

\item
 modal operators:
$\Diamond$, $\Box$.
\end{enumerate}
\end{definition}

\begin{definition}
The set $\Phi$ of \emph{positive modal formulas} is 
the least set that satisfies:

\begin{enumerate}
\item
$\propvar\subseteq\Phi$ and 
$\top$ and $\bot$ are in $\Phi$;

\item 
if $\phi$ and $\psi$ are in $\Phi$
then $(\phi\lor\psi)\in \Phi$ and 
$(\phi\land\psi)\in \Phi$;

\item 
if $\phi\in\Phi$ 
then $(\Diamond\phi)$ and 
$(\Box\phi)$  are in $\Phi$. 
\end{enumerate}
For each positive modal formula $\phi$ and $n\in\omega$, 
$\Box^{0}\phi$ denotes $\phi$ and $\Box^{n+1}\phi$ denotes $\Box(\Box^{n}\phi)$. 
We define $\Diamond^{n}\phi$ in the same way. 
A \emph{positive modal sequent}  is an expression $\phi\db\psi$, where 
$\phi$ and $\psi$ are positive modal formulas. 
\end{definition}

\begin{definition}\label{algmodel}
An \emph{algebraic model} for positive modal formulas
is a pair $\langle A,v\rangle$, where $A$ is a modal 
distributive lattice and $v$ is a function from the set $\propvar$
of propositional variables to $A$. 
The function $v$ is recursively extended to the set of 
all positive modal formulas as follows: 
\begin{enumerate}
\item
$v(\bot)=0$, $v(\top)=1$; 

\item
$v(\phi\lor\psi)=v(\phi)\sqcup v(\psi)$, 
$v(\phi\land\psi)=v(\phi)\sqcap v(\psi)$;

\item
$v(\Diamond\phi)=\Diamond v(\phi)$, 
$v(\Box\phi)=\Box v(\phi)$.  
\end{enumerate}

Let 
$A$ be a modal distributive lattice. 
For every positive modal formula $\phi$, 
we write 
$
A\vl\phi
$,   
if 
$v(\phi)=1$ for every $v:\propvar\yy A$. 
Let $C$ be a class of modal distributive lattices. 
We write
$
C\vl\phi
$, 
if 
$
A\vl\phi
$ 
for every $A\in C$. 
For every positive modal sequent $\phi\db\psi$, we write
$
A\vl\phi\db\psi
$,   
if 
$v(\phi)\leq v(\psi)$ for every $v:\propvar\yy A$. 
Let $C$ be a class of modal distributive lattices. 
We write
$
C\vl\phi\db\psi
$, 
if 
$
A\vl\phi\db\psi
$ 
for every $A\in C$. 
\end{definition}

Semantically, any sequent of positive modal formulas is 
equivalent to an identity of modal distributive lattices. 
More explicitly, we have the following:

\begin{theorem}\label{formulaidentity}
Let $\rho$ be a function from the set of positive modal formulas 
to the set of terms of modal distributive lattices 
that satisfies the following: 
\begin{enumerate}
\item
$\rho$ is an injection from the set $\propvar$ of propositional variables to 
the set $X$ of variables;
\item
$\rho(\bot)=\lanzero$, 
$\rho(\top)=\lanone$; 
\item
$\rho(\phi\lor \psi)=\lansqcup(\rho(\phi),\rho(\psi))$, 
$\rho(\phi\land \psi)=\lansqcap(\rho(\phi),\rho(\psi))$; 
\item
$\rho(\Diamond\phi)=\landiamond(\rho(\phi))$, 
$\rho(\Box\phi)=\lanbox(\rho(\phi))$.  
\end{enumerate}
Then, for every modal distributive lattice $A$ and 
all positive modal formulas $\phi$ and $\psi$,  
\begin{enumerate}
\item
$A\vl\phi$ if and only if $A\vl\rho(\phi)\laneq\lanone$; 
\item
$A\vl\phi\db\psi$ if and only if 
$A\vl\lansqcup(\rho(\phi),\rho(\psi))\laneq\rho(\psi)$. 
\end{enumerate}
\end{theorem}

\begin{proof}
Let $\phi$ be a positive modal formula.
Suppose that $p_{1},\ldots,p_{n}$ is the list of all propositional variables 
that occur in $\phi$. 
Then, it is straightforward to prove that 
$v(\phi)=\rho(\phi)^{A}(v(p_{1}),\ldots,v(p_{n}))$
for every $v:\propvar\yy A$. 
\end{proof}

Now, we introduce a sequent-style proof system $\psgld$.

\begin{definition}
The axioms and the inference rules of $\psgld$ are the following: 
\begin{enumerate}
\item
Reflexivity: 
$\phi\db\phi$;

\item
Transitivity: 
$
\dfrac
{\phi\db\psi,\ \psi\db\chi}
{\phi\db\chi};
$

\item
Disjunction Introduction: 
$\phi\db\phi\lor\psi$, $\psi\db\phi\lor\psi$; 

\item
Disjunction Elimination: 
$\dfrac
{\phi\db\chi,\ \psi\db\chi}{\phi\lor\psi\db\chi}
$;

\item
Conjunction Elimination:
$\phi\land\psi\db\phi$, $\phi\land\psi\db\psi$;

\item
Conjunction Introduction: 
$\dfrac
{\chi\db\phi,\ \chi\db\psi}{\chi\db\phi\land\psi}
$;

\item
Distribution: 
$\phi\land(\psi\lor\chi)\db(\phi\land\psi)\lor(\phi\land\chi)$;

\item
Becker's Rules: 
$
\dfrac{\phi\db\psi}{\Diamond\phi\db\Diamond\psi}
$, 
$
\dfrac{\phi\db\psi}{\Box\phi\db\Box\psi}
$; 

\item
Linearity: 
$\Diamond(\phi\lor\psi)\db\Diamond\phi\lor\Diamond\psi$, 
%
%
%
$\Box\phi\land\Box\psi\db\Box(\phi\land\psi)$; 
\item
$\Box$-$\Diamond$ Interaction: 
$\Diamond\phi\land\Box\psi\db\Diamond(\phi\land\psi)$, 
$\Box(\phi\lor\psi)\db\Box\phi\lor\Diamond\psi$; 

\item
Bottom: $\bot\db\phi$;

\item
Top: $\phi\db\top$; 

\item
Possibilization: 
$\Diamond\bot\db\bot$;

\item 
Necessitation: $\top\db\Box\top$;

\item
Axiom 4: 
$\Diamond\Diamond \phi\db\Diamond\phi$, 
$\Box\phi\db\Box\Box\phi$;

\item
GLD-Rules: 
$
\dfrac
{\phi\db\psi\lor \Diamond^{n}\top \text{ ($\forall n\in\omega$) }}
{\phi\db\psi}
$,
$
\dfrac
{\psi\land\Box^{n}\bot\db\phi\ \text{ ($\forall n\in\omega$) }}
{\psi\db\phi}
$.
\end{enumerate}
\end{definition}

The proof system $\psgld$ is based on Dunn's system 
$\psdunn$ \cite{dnn95}, which consists of the axioms and 
inference rules (1)-(14) above. 
For these axioms and rules, 
we follow Dunn's terminology. 
Each GLD-rule has countably many premises.

\begin{lemma}\label{lemlind}
Let $\cong$ be 
an equivalence relation on the set $\Phi$ of all positive modal formulas 
such that 
$\phi\cong\psi$ if and only if $\phi\db\psi$ and $\psi\db\phi$ are provable in 
$\psgld$.  
Then, 
$A=\langle \Phi/{\cong};\lor,\land,\Diamond,\Box,0,1\rangle$ 
is a GLD-lattice, 
in which operators $0$, $1$, $\lor$, $\land$, 
$\Diamond$ and $\Box$ on $\Phi/{\cong}$ 
are defined by the corresponding logical symbols of representatives. 
\end{lemma}

\begin{proof}
Let $\phi$ and $\psi$ be any positive modal formulas and let 
$[\phi]$ and $[\psi]$ be the equivalence classes 
in $\Phi/{\cong}$, respectively. 
Then, 
\begin{align*}
\text{$\phi\db\psi$ is provable in $\psgld$}
&\eq
\text{$\phi\db\phi\land\psi$ is provable in $\psgld$}\\
&\eq
\phi\land\psi\cong\phi\\
&\eq
[\phi]\sqcap[\psi]=[\phi\land\psi]=[\phi]\\
&\eq
[\phi]\leq[\psi]. 
\end{align*}
Hence, 
\begin{equation}\label{lindorder}
\text{
$\phi\db\psi$ is provable in $\psgld$ 
$\eq$
$[\phi]\leq[\psi]$ in $A$
}
\end{equation}
for any positive modal formulas $\phi$ and $\psi$. 
By \eqref{lindorder}, it is clear that $A$ is a 
modal distributive lattice
that satisfies 
\eqref{gldltransitivity} of Definition \ref{defgldl}. 
We show 
$$
\bigsqcap_{n\in\omega}\left([\phi]\sqcup\Diamond^{n}[\top]\right)= [\phi]
$$
for every positive modal formula $\phi$. 
By Disjunction Introduction, $\phi\db\phi\lor\Diamond^{n}\top$ 
is provable in $\psgld$ 
for all $n\in\omega$. 
Hence, 
$[\phi]\leq[\phi]\sqcup\Diamond^{n}[\top]$ 
by \eqref{lindorder}. 
Suppose that a positive modal formula $\psi$ satisfies 
$[\psi]\leq[\phi]\sqcup\Diamond^{n}[\top]$ for all $n\in\omega$. 
Then, by \eqref{lindorder}, 
$\psi\db\phi\lor \Diamond^{n}\top$ is provable 
in $\psgld$ 
for all $n\in\omega$.  
Hence, by the GLD-rule, $\psi\db\phi$ is provable in $\psgld$. 
Therefore, 
$[\psi]\leq[\phi]$ by \eqref{lindorder}.  
The other equality of  
\eqref{ioriand} of Definition \ref{defgldl}
is proved in the same way. 
\end{proof}

Now, we show that the proof system $\psgld$ is sound and complete with respect to 
the class $\cgld$ of all GLD-lattices.

\begin{theorem} \label{psgld-cgld-completeness}
For all positive modal formulas  $\phi$ and $\psi$, 
$\phi\db\psi$ is provable in $\psgld$ if and only if 
$\cgld\vl\phi\db\psi$. 
\end{theorem}

\begin{proof}
Soundness is proved by induction on the height of the derivations 
in $\psgld$. 
We show completeness. 
Suppose $\phi\db\psi$ is not provable in $\psgld$. Then, 
$[\phi]\not\leq[\psi]$ in the modal distributive lattice $A$ given in 
Lemma \ref{lemlind}. 
Define $v:\propvar\yy A$ by $v(p)=[p]$ for every propositional 
variable $p$. It is straightforward to show that $v(\phi)=[\phi]$ for 
every positive modal formula $\phi$. 
Hence, $A\not\vl\phi\db\psi$. 
Therefore, $\cgld\not\vl\phi\db\psi$ by Lemma \ref{lemlind}. 
\end{proof}

The following is the main result of this paper.

\begin{theorem}\label{psgldcompleteness}
Let $\phi$ and $\psi$ be any positive modal formulas. 
Then, $\phi\supset\psi\in\propgl$
if and only if 
$\phi\db\psi$ is provable in $\psgld$.  
\end{theorem}

\begin{proof}
\begin{align*}
&\phi\supset\psi\in\propgl\\
&\eq
\forall A\in\cglmd\forall v:\propvar\yy A
\left(
v(\phi\supset\psi)=1
\right)
&
(\text{Theorem \ref{propglcompletenessglda}})\\
&\eq
\cglmd\reduct_{\Sigma_{\cmdl}}\vl \phi\db\psi
&
(\text{Definition \ref{algmodel}})\\
&\eq
\cglmd\reduct_{\Sigma_{\cmdl}}
\vl\lansqcup(\rho(\phi),\rho(\psi))\laneq\rho(\psi)
&
(\text{Theorem \ref{formulaidentity}})\\
&\eq
\cgld\vl\lansqcup(\rho(\phi),\rho(\psi))\laneq\rho(\psi)
&
(\text{Theorem \ref{qiequivalence}})\\
&\eq
\cgld\vl\phi\db\psi
&
(\text{Theorem \ref{formulaidentity}})\\
&\eq
\text{
$\phi\db\psi$ is provable in $\psgld$
}.
&
(\text{Theorem \ref{psgld-cgld-completeness}})
\end{align*}
\end{proof}

\begin{definition}
Let $Z$ be a Kripke frame and  $\phi\db\psi$ be a positive modal sequent. 
We say that $\phi\db\psi$ is \emph{valid} in $Z$ if $\phi\supset\psi$ is 
valid in $Z$. 
\end{definition}

By Theorem \ref{psgldcompleteness} and Kripke completeness of 
$\propgl$, we have the following.

\begin{corollary}
$\psgld$ is sound and complete with respect to the class 
of transitive and conversely well-founded Kripke frames. 
\end{corollary}

\section{Algebraic proof of Kripke completeness for
other positive modal fragments}\label{sec:applications}

In this section, we give
algebraic proofs of completeness for the positive fragments of 
several other modal logics,  
thereby providing 
explicit proofs of 
some results
noted in Dunn \cite{dnn95}. 
The method is the same as that used for $\psgld$, except that here 
we use prime filters instead of GLD-filters. 
For each modal distributive lattice $A$ and $x\in A$, 
we write $\pflt(A)$ 
for the set of 
all prime filters of $A$, 
and write 
$\etap(x)$ for the set of all prime filters of $A$ containing $x$.

\begin{definition}\label{primefltfrm}
Let $A$ be a modal distributive lattice. 
We say that $Z=\langle\pflt(A),\nsystem\rangle$
is a \emph{prime filter frame} of $A$ if $\nsystem(F)$ is a 
filter of $\power(\pflt(A))$ and 
satisfies the conditions 
\eqref{etasubset} and \eqref{etaemptyset}
of Definition \ref{gldfltfrm}, with $\eta$ replaced by $\etap$, 
for every $F\in\pflt(A)$. 
\end{definition}

By the same argument as in the proof of 
Theorem \ref{embedding}, we have the following:

\begin{theorem}\label{pftembedding}
Let $A$ be a modal distributive lattice.  
If $Z=\langle\pflt(A),\nsystem\rangle$ 
is a prime filter frame,
then $\etap:A\yy\dalgb(Z)\reduct_{\Sigma_{\cmdl}}$ is a monomorphism of modal distributive lattices. 
\end{theorem}

For the axioms corresponding to reflexivity and transitivity, 
we have the following analogue of Theorem \ref{cgl-cgld}.

\begin{theorem}\label{cmdl-cma}
Let $A$ be a modal distributive lattice. 
For each $F\in\pflt(A)$,  
define $\nsystem(F)$ by 
\begin{equation}
\nsystem(F)
=
\ua\left\{
\etap(x)\cap\compl{\etap(y)}\mid 
\text{$\Box x\in F$ and $\Diamond y\not\in F$}
\right\}. 
\end{equation}
Then, 
\begin{enumerate}
\item
$Z=\langle \pflt(A),\nsystem\rangle$ 
is a prime filter frame of $A$; 
\item
$\dalgb(Z)$ is a modal algebra; 
\item
if $\Box x\leq x$ and $x\leq\Diamond x$ hold for all $x\in A$, 
then $\Box_{Z} X\leq X$ holds for all $X\in\dalgb(Z)$; 
\item
if $\Box x\leq \Box\Box x$ and $\Diamond\Diamond x\leq\Diamond x$ hold for all $x\in A$, 
then $\Box_{Z} X\leq\Box_{Z}\Box_{Z} X$ holds for all $X\in\dalgb(Z)$. 
\end{enumerate}
\end{theorem}

\begin{proof}
Assertions (1) and (2)
follow by the same argument as in the proof of Theorem \ref{cgl-cgld}, 
using prime filters in place of GLD-filters. 
We show (3). 
Suppose that $F\in\Box_{Z}X$. Then, there exist $\Box x\in F$ and 
$\Diamond y\not\in F$ such that 
\begin{equation}\label{finboxx}
\etap(x)\cap\compl{\etap(y)}\subseteq X. 
\end{equation}
By assumption, $x\in F$ and $y\not\in F$. Hence, $F\in X$. 
We show (4). 
Suppose that $F\in\Box_{Z}X$. Then, there exist $\Box x\in F$ and 
$\Diamond y\not\in F$ such that \eqref{finboxx} holds. 
Since $\etap$ is a homomorphism of modal distributive lattices 
and $\Box_{Z}$ is monotone, 
\begin{equation}
\etap(\Box x)\cap\compl{\etap(\Diamond y)}
=
\Box_{Z}\left(\etap(x)\cap\compl{\etap(y)}\right)
\subseteq 
\Box_{Z}X. 
\end{equation}
By assumption, $\Box\Box x\in F$ and $\Diamond\Diamond y\not\in F$. 
Hence, $F\in\Box_{Z}\Box_{Z}X$. 
\end{proof}

For the axiom corresponding to symmetry, we use a prime filter frame that is a 
symmetric Kripke frame. 

\begin{theorem}\label{symmetry}
Let $A$ be a modal distributive lattice that 
satisfies
\begin{equation}\label{axsymmetry} 
x\leq\Box\Diamond x,\ \ 
\Diamond\Box x\leq x
\end{equation}
for all $x\in A$. 
Define a binary relation $R$ on $\pflt(A)$ 
as follows: 
for all prime filters $F$ and $G$ of $A$, 
$(F,G)\in R$ if and only if 
\begin{equation}\label{condsymmetry}
\Box^{-1}[F]\subseteq G,\ \Diamond^{-1}[\compl{F}]\subseteq\compl{G},\ 
\Box^{-1}[G]\subseteq F,\ \Diamond^{-1}[\compl{G}]\subseteq\compl{F}. 
\end{equation}
Then, 
\begin{enumerate}
\item
$Z=\langle \pflt(A),R\rangle$ 
is a prime filter frame of $A$; 
\item
$\dalgb(Z)$ is a modal algebra such that $X\leq\Box_{Z}\Diamond_{Z}X$ for 
all $X\subseteq\pflt(A)$; 
\item
if $\Box x\leq x$ and $x\leq\Diamond x$ hold for all $x\in A$, 
then $\Box_{Z} X\leq X$ holds for all $X\in\dalgb(Z)$; 
\item
if $\Box x\leq \Box\Box x$ and $\Diamond\Diamond x\leq\Diamond x$ hold for all $x\in A$, 
then $\Box_{Z} X\leq\Box_{Z}\Box_{Z} X$ holds for all $X\in\dalgb(Z)$. 
\end{enumerate}
\end{theorem}

\begin{proof}
(1): 
First, we show that $\etap(\Box x)=\Box_{Z}\etap(x)$. 
It is clear from the definition of $R$ that $\etap(\Box x)\subseteq\Box_{Z}\etap(x)$ for 
every $x\in A$. 
We show the converse. 
Suppose that $F\not\in\etap(\Box x)$. 
Let $I$ be the ideal generated by $x$ and $\Diamond^{-1}[\compl{F}]$. 
We claim that $\Box^{-1}[F]\cap I=\emptyset$. 
If not, there exists $a\in\Diamond^{-1}[\compl{F}]$ such that 
$a\sqcup x\in\Box^{-1}[F]$. 
Then, $\Diamond a\sqcup\Box x\in F$, since  
$\Box(a\sqcup x)\leq\Diamond a\sqcup\Box x$. 
Since $F$ is a prime filter, 
$\Diamond a\in F$ or $\Box x\in F$, which is a contradiction. 
This completes the proof of the claim. 
Hence, by the prime filter theorem, there exists a prime filter $G$ of $A$ such that 
$\Box^{-1}[F]\subseteq G$ and $G\cap I=\emptyset$. 
We show $(F,G)\in R$. 
The first two conditions of \eqref{condsymmetry} follow from the definition of $G$. 
We show $\Box^{-1}[G]\subseteq F$. 
Suppose $a\in\Box^{-1}[G]$. If $a\not\in F$, then 
$\Diamond\Box a\not\in F$, and hence $\Box a\in\Diamond^{-1}[\compl{F}]$.
However, this is a contradiction, since $G$ and $\Diamond^{-1}[\compl{F}]$ are disjoint.  
Lastly, we show $\Diamond^{-1}[\compl{G}]\cap F=\emptyset$. 
Suppose there exists $a\in\Diamond^{-1}[\compl{G}]\cap F$. 
Then $\Diamond a\in\Box^{-1}[F]$, since $a\leq\Box\Diamond a\in F$. 
However, this is a contradiction, since $\Diamond a\in \compl{G}$ 
and $\Box^{-1}[F]\subseteq G$.

Next, we show that $\etap(\Diamond x)=\Diamond_{Z}\etap(x)$. 
It is clear from the definition of $R$ that 
$\Diamond_{Z}\etap(x)\subseteq\etap(\Diamond x)$ for 
every $x\in A$. 
We show the converse. 
Suppose that $F\in\etap(\Diamond x)$. 
Let $G$ be the filter generated by $x$ and $\Box^{-1}[F]$. 
We claim that $\Diamond^{-1}[\compl{F}]\cap G=\emptyset$. 
If not, there exists $a\in\Box^{-1}[F]$ such that 
$a\sqcap x\in\Diamond^{-1}[\compl{F}]$. 
Then, $\Box a\sqcap\Diamond x\in\compl{F}$, since  
$\Diamond(a\sqcap x)\geq\Box a\sqcap\Diamond x$. 
Since $\Diamond x\in F$ and $\Box a\in F$, 
this is a contradiction. 
This completes the proof of the claim. 
Hence, by the prime filter theorem, there exists a prime filter $H$ of $A$ such that 
$G\subseteq H$ and $H\cap\Diamond^{-1}[\compl{F}]=\emptyset$. 
We show $(F,H)\in R$. 
The first two conditions of \eqref{condsymmetry} follow from the definition of $H$. 
We show $\Box^{-1}[H]\subseteq F$. 
Suppose $a\in\Box^{-1}[H]$. If $a\not\in F$, then 
$\Diamond\Box a\not\in F$, and hence $\Box a\in\Diamond^{-1}[\compl{F}]$.
However, this is a contradiction, since $H$ and $\Diamond^{-1}[\compl{F}]$ are disjoint.  
Lastly, we show $\Diamond^{-1}[\compl{H}]\cap F=\emptyset$. 
Suppose there exists $a\in\Diamond^{-1}[\compl{H}]\cap F$. 
Then $\Diamond a\in\Box^{-1}[F]$, since $a\leq\Box\Diamond a\in F$. 
However, this is a contradiction, since $\Diamond a\in \compl{H}$ 
and $\Box^{-1}[F]\subseteq H$.

(2): 
It is clear from the definition of $R$ that $Z$ is a symmetric Kripke frame. 
Hence, $X\leq\Box_{Z}\Diamond_{Z}X$ for 
all $X\in\dalgb(Z)$. 

(3):
It is sufficient to prove that 
$R$ is reflexive. 
Let $F\in\pflt(A)$. Then, $\Box^{-1}[F]\subseteq F$, since 
$\Box x\leq x$ for all $x\in A$, and 
$\Diamond^{-1}[\compl{F}]\subseteq\compl{F}$, since
$x\leq\Diamond x$ for all $x\in A$. 
Hence, $R$ is reflexive.

(4):
It is sufficient to prove that 
$R$ is transitive. 
For every $F\in\pflt(A)$, 
\begin{equation}\label{filtertransitivity}
\Box^{-1}[F]\subseteq\Box^{-1}\left[\Box^{-1}[F]\right],\ \
\Diamond^{-1}[\compl{F}]\subseteq\Diamond^{-1}\left[\Diamond^{-1}[\compl{F}]\right]
\end{equation}
holds by assumption. 
Suppose that $(F_{1},F_{2})\in R$ and $(F_{2},F_{3})\in R$. 
Then, 
\begin{equation}\label{tran1}
\Box^{-1}[F_{1}]\subseteq F_{2},\ \Diamond^{-1}[\compl{F_{1}}]\subseteq\compl{F_{2}},\ 
\Box^{-1}[F_{2}]\subseteq F_{1},\ \Diamond^{-1}[\compl{F_{2}}]\subseteq\compl{F_{1}}
\end{equation}
and 
\begin{equation}\label{tran2}
\Box^{-1}[F_{2}]\subseteq F_{3},\ \Diamond^{-1}[\compl{F_{2}}]\subseteq\compl{F_{3}},\ 
\Box^{-1}[F_{3}]\subseteq F_{2},\ \Diamond^{-1}[\compl{F_{3}}]\subseteq\compl{F_{2}}. 
\end{equation}
We show $(F_{1},F_{3})\in R$. 
By \eqref{filtertransitivity} and the first inclusions in \eqref{tran1} and \eqref{tran2}, 
we have 
$$
\Box^{-1}[F_{1}]
\subseteq
\Box^{-1}\left[\Box^{-1}[F_{1}]\right]
\subseteq
\Box^{-1}[F_{2}]
\subseteq
F_{3}.
$$
The remaining inclusions follow similarly. 
\end{proof}

Let $\gm$ be a set of positive modal sequents. 
We write $\psdunn(\gm)$ for the proof system obtained from
Dunn's proof system $\psdunn$ by adding all positive modal sequents in $\gm$ as axioms. 
We write $\cma(\gm)$ for the class of all modal algebras 
that validate $\gm$, 
$\cmdl(\gm)$ for the class of all modal distributive lattices 
that validate $\gm$, 
and 
$\kfrm(\gm)$ for the class of all Kripke frames that validate $\gm$.

The preceding results lead to the following representation theorem.

\begin{theorem}\label{otherembedding}
Let 
$S_{1}=\{p\db\Diamond p,\Box p\db p\}$, 
$S_{2}=\{\Diamond\Diamond p\db \Diamond p, \Box p\db \Box\Box p\}$, 
and 
$S_{3}=\{\Diamond\Box p\db p, p\db \Box\Diamond p\}$.
Let $\gm$ be the union of any subset of $\{S_{1},S_{2},S_{3}\}$.  
Then, $\cmdl(\gm)\subseteq\cliso\clsub(\cma(\gm)\reduct_{\Sigma_{\cmdl}})$. 
\end{theorem}

\begin{proof}
This follows from Theorems \ref{pftembedding} and \ref{cmdl-cma} 
when $S_{3}\not\subseteq\gm$, 
and 
from Theorems \ref{pftembedding} and  \ref{symmetry} 
when $S_{3}\subseteq\gm$.  
\end{proof}

The representation theorem immediately yields the following consequence.

\begin{corollary}\label{otherqi}
Let $\gm$ be any set of positive modal sequents given in Theorem \ref{otherembedding}. 
For any quasi-identity $\alpha$ of modal distributive lattices, 
$\cma(\gm)\reduct_{\Sigma_{\cmdl}}\vl\alpha$ if and only if 
 $\cmdl(\gm)\vl\alpha$. 
\end{corollary}

\begin{proof}
The proof is similar to that of Theorem \ref{qiequivalence}, but 
we use Theorem \ref{otherembedding} instead of Theorem \ref{cgldembcgl}. 
\end{proof}

This, in turn, yields the following results.

\begin{corollary}\label{gmconservative}
Let $\gm$ be any set of positive modal sequents given in Theorem \ref{otherembedding}. 
For all positive modal formulas $\phi$ and $\psi$, 
$\phi\db\psi$ is provable in 
$\psdunn(\gm)$ 
if and only if $\phi\supset\psi$ is in 
the least extension of $\logic{K}$ containing $\gm$. 
\end{corollary}

\begin{proof}
The proof is similar to that of Theorem \ref{psgldcompleteness}, but 
we use Corollary \ref{otherqi} instead of Theorem \ref{qiequivalence}. 
\end{proof}

\begin{corollary}[Dunn \cite{dnn95}]
Let $\gm$ be any set of positive modal sequents given in Theorem \ref{otherembedding}. 
Then $\psdunn(\gm)$ is sound and complete with respect to $\kfrm(\gm)$. 
\end{corollary}

\begin{proof}
By Corollary \ref{gmconservative}.  
\end{proof}

We conclude this section by presenting an example of Kripke incomplete 
positive modal logic.

\begin{theorem}\label{incompleteness}
There exists a finite set $\Sigma$ of positive modal sequents and 
a positive modal sequent $\phi\db\psi$ which satisfies the following conditions:
\begin{enumerate}
\item
$\Sigma$ is closed under $\Box$-$\Diamond$ dual;
\item
$\kfrm(\Sigma)\vl\phi\db\psi$, 
but 
$\phi\db\psi$ is not provable in $\psdunn(\Sigma)$. 
\end{enumerate}
\end{theorem}

\begin{proof}
Let 
\begin{equation*}
\Sigma
=
\{
\Diamond\phi\land\Diamond\psi\db\Diamond(\phi\land\psi), 
\Box(\phi\lor\psi)\db\Box\phi\land\Box\psi
\}. 
\end{equation*}
It is clear that $\psdunn(\Sigma)$ is sound and complete 
with respect to $\cmdl(\Sigma)$. 
Let $S=\{0,a,1\}$ be a modal distributive lattice based on the three-element chain
such that $\Box$ is the identity map and $\Diamond 0=0$ and 
$\Diamond 1=\Diamond a=1$. 
Then, $S\in\cmdl(\Sigma)$ and $\Diamond a=1\not\leq a=\Box a$ 
hold in $S$. 
Hence, 
$
\cmdl(\Sigma)
\not\vl
\Diamond x\leq \Box x.
$
Therefore, $\Diamond p\db\Box p$ is not provable in $\psdunn(\Sigma)$. 
However, for every variable $x$
$$
\cma(\Sigma)
\models \Diamond x\leq \Box x.
$$
Hence, $\logic{L}$ is incomplete with respect to $\cma(\Sigma)$. 
Therefore, $\logic{L}$ is incomplete with respect to $\kfrm(\Sigma)$. 
\end{proof}

\section{Quasi-sequents for positive modal logic}\label{sec:quasi-sequent}

In this section, we extend the completeness theorem 
previously established for positive modal sequents to the setting of 
quasi-sequents, that is, the logical counterpart of quasi-identities. 
We write $\gm$ for any set of positive modal sequents given
in Theorem \ref{otherembedding}. 

\begin{definition}
A {\em quasi-sequent} for positive modal logic is an expression of the form 
\begin{equation}\label{hypersq}
\phi_{1}\db\psi_{1},\ldots,\phi_{n}\db\psi_{n}\thn
\phi\db\psi,  
\end{equation}
where
$\phi_{1}\db\psi_{1},\ldots,\phi_{n}\db\psi_{n}$ and 
$\phi\db\psi$ are positive modal sequents. 
\end{definition}

\begin{definition}
Let $\alpha$ be a positive modal quasi-sequent of the form \eqref{hypersq}. 
For each modal distributive lattice $A$, 
$\alpha$ is said to be {\em valid} in $A$ if and only if for every 
$v:\propvar\yy A$, 
if $v(\phi_{i})\leq v(\psi_{i})$ for every $i=1,\ldots,n$, then
$v(\phi)\leq v(\psi)$. 
If this is the case, we write $A\vl\alpha$. 
Let $\alpha$ be a quasi-sequent and let
$C$ be a class of modal distributive lattices. 
We write $C\vl\alpha$ if $A\vl\alpha$ for every $A\in C$. 
\end{definition}

\begin{theorem}
For any positive modal quasi-sequent $\alpha$, 
$\cgld\vl\alpha$ 
($\cmdl(\gm)\vl\alpha$, respectively)
if and only if $\cglmd\vl\alpha$
($\cma(\gm)\vl\alpha$, respectively). 
\end{theorem}

\begin{proof}
By Theorem \ref{malcevth} and Theorem \ref{cgldembcgl}
(Theorem \ref{otherembedding}, respectively). 
\end{proof}

\begin{definition}
A positive modal quasi-sequent 
$$
\phi_{1}\db\psi_{1},\ldots,\phi_{n}\db\psi_{n}\thn
\phi\db\psi. 
$$
is said to be {\em provable} in $\psgld$ ($\psdunn(\gm)$, respectively) 
if $\phi\db\psi$ is 
derivable in the proof system obtained by adding 
$\phi_{1}\db\psi_{1},\ldots,\phi_{n}\db\psi_{n}$
as axioms to $\psgld$ ($\psdunn(\gm)$, respectively). 
\end{definition}

The following is an algebraic completeness theorem for
positive modal quasi-sequents.

\begin{theorem}\label{qsalgcmp}
A positive modal quasi-sequent $\alpha$ is provable in $\psgld$ 
($\psdunn(\gm)$, respectively)
if 
and only if 
$\cglmd\vl\alpha$ 
($\cma(\gm)\vl\alpha$, respectively). 
\end{theorem}

\begin{proof}
Let $\alpha$ be a positive modal sequent of the form
\eqref{hypersq} . 
Soundness follows by induction 
on the height of the derivations 
in 
the proof system obtained by adding 
$\phi_{1}\db\psi_{1},\ldots,\phi_{n}\db\psi_{n}$
as axioms to $\psgld$ ($\psdunn(\gm)$, respectively). 
For completeness, 
consider the  
equivalence relation 
$\cong$ 
on the set $\Phi$ of all positive modal formulas 
such that 
$\phi\cong\psi$ if and only if $\phi\db\psi$ and $\psi\db\phi$ are provable in 
the proof system obtained by adding 
$\phi_{1}\db\psi_{1},\ldots,\phi_{n}\db\psi_{n}$
as axioms to $\psgld$ ($\psdunn(\gm)$, respectively). 
\end{proof}

We write $\frmgld$ for the class of Kripke frames $Z$ such that $\dalgb(Z)\in\cglmd$.

\begin{definition}
A positive modal quasi-sequent $\alpha$ is said to be {valid} in a Kripke frame
$Z$ if $\dalgb(Z)\vl\alpha$. 
\end{definition}

The following is a Kripke completeness theorem 
for positive modal quasi-sequents.

\begin{theorem}
A quasi-sequent $\alpha$ is provable in $\psgld$ 
($\psdunn(\gm)$, respectively)
if 
and only if $\frmgld\vl\alpha$
($\kfrm(\gm)\vl\alpha$, respectively). 
\end{theorem}

\begin{proof}
Soundness is trivial. 
We show completeness. 
Suppose that $\alpha$ is not provable in $\psgld$. 
By Theorem \ref{qsalgcmp}, there exists $A\in\cglmd$ such that $A\not\vl\alpha$. 
It is proved in Theorem 3.9 of Sasaki and Tanaka \cite{ssk-tnk24}
that for every GLD-algebra $A$ there exists a 
Kripke frame $Z\in\frmgld$ such that 
$A$ embeds into $\dalgb(Z)$. 
Hence, $\dalgb(Z)\not\vl\alpha$. 
For the case where $\alpha$ is not provable in $\psdunn(\gm)$, 
we appeal 
canonicity of $\gm$ instead of Theorem 3.9 of Sasaki and Tanaka \cite{ssk-tnk24}.  
\end{proof}

\section*{Declaration on the Use of AI models}

AI models were used to assist in improving the English presentation 
of this paper and in proving 
Theorem \ref{incompleteness}. 
All mathematical arguments and AI-generated suggestions were 
checked and reviewed by the author, 
who takes full responsibility for the 
content of the paper.

\bibliographystyle{plain}
\bibliography{myref.sjis}

\end{document}